\documentclass[11pt]{article}

\usepackage[margin=1in]{geometry}
\usepackage[T1]{fontenc}
\usepackage[utf8]{inputenc}
\usepackage{amsmath,amssymb,amsfonts,amsthm}
\usepackage{graphicx}
\usepackage{hyperref}
\usepackage{enumitem}

\newcommand{\keywords}[1]{\par\medskip\noindent\textbf{Keywords: }#1}
\newcommand{\msc}[1]{\par\noindent\textbf{MSC 2020: }#1}

\title{A law of thin processes with neighbour-count thinning}

\author{Kateryna Hlyniana\\
\small School of Mathematics, Jilin University, Changchun 130061, China\\
\small Institute of Mathematics, National Academy of Sciences of Ukraine, Kyiv 01024, Ukraine\\
\small \texttt{glinkate@gmail.com}}
\date{}

\newtheorem{theorem}{Theorem}
\newtheorem{lemma}{Lemma}

\theoremstyle{remark}
\newtheorem{remark}{Remark}

\begin{document}
\maketitle
\begin{abstract}
We consider the superposition $S_n$ of $n$ i.i.d.\ simple point processes on $\mathbb{R}^d$ and apply a dependent thinning $T_n$, where the retention probability of a point depends on its local neighbour count within a radius $r_n$. While superpositions under independent thinning converge to Poisson processes, we show that under a \emph{critical geometric scaling} $n v_d r_n^d \to \tau \in (0,\infty)$,  the local interactions are transformed in the limit into an inhomogeneous Poisson point process with modified intensity. We prove that the thinned sequence $T_n(S_n)$ converges to a Poisson process with a non-linearly modified intensity $\tilde\lambda(x) = \lambda(x)\alpha(\tau\lambda(x))$.
\end{abstract}
\msc{60G55; 60F05}

\section*{1.\ Introduction}

The convergence of superpositions of independent sparse point processes to a Poisson limit is a classical result in stochastic geometry \cite{Grigelionis1963Poisson}. Recently, Aldridge \cite{Aldridge2026ThinProcesses} unified these results under the ``Law of Thin Processes,'' providing a transparent proof that the superposition of $n$ i.i.d.\ processes, thinned independently by $1/n$, converges to a Poisson process.  Aldridge's approach highlights the usefulness of working with the probability generating functional evaluated at $v=1-u$, which is particularly convenient for thinning arguments.

However, many physical and biological systems exhibit retention mechanisms that are explicitly interaction-dependent. For example, in forestry or cellular biology, points may survive only when the local neighbourhood is not too crowded.   In such situations, the thinning probability is no longer a constant, but depends on the local configuration of points. In our setting, a point with $k$ neighbours within distance $r_n$ is retained with probability $q_n(k)$. We shall work in the regime
\(
 q_n(k)=\frac1n a(k),
\)
 where the factor $1/n$ gives the rare-thinning scale and the bounded function $a(\cdot)$ records how retention depends on the neighbour count. Thus the model reduces to the classical law of thin processes only in the special case where $a(\cdot)\equiv 1$.

In this note, we study the superposition $S_n$ of $n$ i.i.d.\ copies of a point process $X$ under this neighbour-count thinning. The relevant local quantity is the number of points within distance $r_n$ of a typical point. Under the critical scaling
\(
n v_d r_n^d \;\longrightarrow\; \tau \in (0,\infty),
\)
 this local count remains of order one and is asymptotically Poisson with mean depending on $\tau$ and the local intensity $\lambda(x)$. The limiting retention at location $x$ is therefore obtained by averaging the rule $a(\cdot)$ against this Poisson law; this Poisson average is encoded later by the function $\alpha$. If $r_n$ decays faster than this rate ($\tau=0$), the local neighbourhoods become empty and the dependence disappears at leading order. If $r_n$ decays slower ($\tau=\infty$), the local counts diverge. It is only in the critical regime  that the microscopic local interactions  are transformed in the scaling limit into an inhomogeneous Poisson field with a modified intensity $\tilde\lambda(x)$.

 Dependent thinning of point processes is classical in spatial statistics, for example in Mat\'ern-type hard-core models and their generalizations \cite{TeichmannBallaniBoogaart2013}. Serfozo \cite{Serfozo1984} proved convergence results for sums of thinned point processes under a general sparsity framework, while Schuhmacher \cite{Schuhmacher2009} derived quantitative Poisson approximation bounds for dependent thinnings. In contrast, we consider a superposition of i.i.d.\ point processes under local neighbour-count thinning with shrinking interaction radius, and obtain the explicit limiting intensity.
To the best of our knowledge, such an explicit limit theorem is not available in the existing literature.

\section*{2.\ Setup and Main Result}

Let $X$ be a simple point process on $\mathbb{R}^d$, viewed as a random counting measure (see, e.g., \cite[Chs.~1--3]{LastPenrose2018} for this viewpoint and basic notation).
Thus for a Borel set $B\subset\mathbb{R}^d$,
\(
X(B):=\#(X\cap B)
\)
denotes the number of points of $X$ in $B$.
Assume that $X$ has intensity measure $\mu$, absolutely continuous with respect to Lebesgue measure,
\(
\mu(B):=\mathbb{E}X(B)=\int_B \lambda(x)\,dx,
\)
where $\lambda:\mathbb{R}^d\to[0,\infty)$ is locally bounded (and later assumed continuous).  The absolute continuity assumption is used to express the local scaling in terms of Euclidean ball volumes and to obtain the explicit pointwise form of the limiting intensity. A more general measure-theoretic formulation may be possible, but is not pursued in this short note.
Write $B(x,r):=\{y\in\mathbb{R}^d:\|y-x\|\le r\}$ for the closed ball of radius $r$ centred at $x$,
and $v_d:=|B(0,1)|$ for the volume of the unit ball in $\mathbb{R}^d$.

For a locally finite (deterministic) configuration $\varphi\subset\mathbb{R}^d$ and $x\in\varphi$, define
the neighbour count within radius $r>0$ by
\[
N_r(x,\varphi):=\#\big(\varphi\cap B(x,r)\setminus\{x\}\big).
\]

Given a function $q:\mathbb{N}_0\to[0,1]$ and a radius $r>0$, define the \emph{neighbour-count thinning}
$T(q,r;\varphi)$ as follows: conditional on $\varphi$, retain each point $x\in\varphi$ {independently}
with probability $q(N_r(x,\varphi))$.

Let $X_1,X_2,\dots$ be i.i.d.\ copies of $X$ and set $S_n:=X_1+\cdots+X_n$.
Fix sequences $(r_n)_{n\ge1}$ with $r_n\downarrow0$ and $(q_n)_{n\ge1}$ of retention rules, and define
\[
Y_n := T_n(S_n) := T(q_n,r_n;S_n).
\]

For $u\in\mathcal V$ (measurable $u:\mathbb R^d\to[0,1]$ with bounded support),  define
\[
A_\xi(u):=\mathbb E\prod_{x\in\xi}(1-u(x)).
\]
 Following \cite{Aldridge2026ThinProcesses}, we shall refer to $A_\xi$ as the \emph{alternate}
probability generating functional in this paper. It is simply the usual probability generating
functional evaluated at $1-u$; for background on the usual p.g.fl.\ see
\cite{DaleyVereJones2008Vol2,LastPenrose2018}.

\paragraph{Second-order regularity  assumption}.
We assume that $X$ has a locally bounded second-order product density $\rho^{(2)}(x,y)$ off the diagonal,
i.e.\ for bounded Borel sets $B_1,B_2$,
\[
\mathbb{E}\sum_{x\neq y\in X}\mathbf{1}_{B_1}(x)\mathbf{1}_{B_2}(y)
=\int_{B_1}\int_{B_2}\rho^{(2)}(x,y)\,dx\,dy,
\]
with $\sup_{(x,y)\in K^2,\,x\neq y}\rho^{(2)}(x,y)<\infty$ for each compact $K\subset\mathbb{R}^d$.

Here and below, for $m\ge1$ and $k\in\mathbb N_0$ we write
$(k)_m:=k(k-1)\cdots(k-m+1)$ for the falling factorial, in particular $(k)_2=k(k-1)$.
For a Borel set $B$, $(X(B))_2:=X(B)(X(B)-1)$, so $\mathbb E[(X(B))_2]$ is the second factorial moment
of the count $X(B)$. Note that our assumption on $\rho^{(2)}$ implies the usual small-ball bound
\begin{equation}
\label{eq:smallball}
\sup_{x\in K}\mathbb{E}\big(X(B(x,r))\big)_2
=\sup_{x\in K}\int_{B(x,r)}\int_{B(x,r)}\rho^{(2)}(y,z)\,dy\,dz
=O(r^{2d})
\quad(r\downarrow0),
\end{equation}
uniformly on compact  sets $K$.

We work in the ``rare retention'' regime
\begin{equation}
\label{eq:rare}
q_n(k)=\frac{1}{n}a(k),\qquad k\in\mathbb{N}_0,
\end{equation}
for some bounded function $a:\mathbb{N}_0\to[0,\infty)$ (so $\sup_k a(k)=:a_\ast<\infty$), and in the
local scaling
\begin{equation}
\label{eq:crit}
n v_d r_n^d \longrightarrow \tau\in[0,\infty).
\end{equation}

Define, for $t\ge0$,
\[
\alpha(t):=\mathbb{E}\big[a(\mathrm{Poi}(t))\big],
\]
and the candidate limiting intensity density
\(
\tilde\lambda(x) := \lambda(x)\,\alpha \big(\tau\lambda(x)\big).
\)

\begin{theorem}[Poisson limit for neighbour-count thinnings of superpositions]
\label{thm:main}
 Assume that $X$ is simple, satisfies the second-order regularity assumption above,
and has continuous strictly positive intensity density $\lambda$.  Let $Y_n=T(q_n,r_n;S_n)$ with $q_n$ and $r_n$ satisfying \eqref{eq:rare}--\eqref{eq:crit}.
Then $Y_n$ converges weakly, as $n\to\infty$, to an inhomogeneous Poisson point process on $\mathbb{R}^d$
with intensity measure $\tilde\lambda(x)\,dx$.
Equivalently, for every $u\in\mathcal V$,
\[
A_{Y_n}(u)\;\longrightarrow\;\exp \Big(-\int_{\mathbb{R}^d}u(x)\,\tilde\lambda(x)\,dx\Big).
\]
\end{theorem}

\begin{remark}
If $a(\cdot)\equiv 1$, then $q_n\equiv 1/n$ and $\alpha(\cdot)\equiv 1$, hence $\tilde\lambda=\lambda$ and
the limit reduces to the Poisson limit in \cite{Aldridge2026ThinProcesses}. 
If $\tau=0$ (e.g.\ $r_n=o(n^{-1/d})$), then $\alpha(\tau\lambda(x))=\alpha(0)=a(0)$ and the neighbour
dependence disappears at leading order.
\end{remark}
 The effect of the transformation $\tilde\lambda(x)=\lambda(x)\alpha(\tau\lambda(x))$ is illustrated by concrete examples in Section~4 below.
\section*{3.\ Proofs}
 At a high level, the proof shows that the alternate p.g.fl.\ of $Y_n$ is asymptotically governed by a single linear statistic. Lemma~\ref{lem:loglin} is the starting point: it gives a log-linear representation of $A_{Y_n}(u)$ and bounds the quadratic remainder. To understand the leading term in that representation, one needs the asymptotic behaviour of the local neighbour count entering the thinning rule; this is provided by Lemma~\ref{lem:localpois}, which shows that the contribution from the other components of the superposition is asymptotically Poisson under the critical scaling. Lemma~\ref{lem:var} then shows that the resulting linear term concentrates around its mean, so that the problem reduces to identifying that mean. Lemma~\ref{lem:mean} computes the limit of the linear part, using the Poissonisation from Lemma~\ref{lem:localpois} together with reduced Palm calculus. Finally, the proof of Theorem~\ref{thm:main} combines the log-linear representation from Lemma~\ref{lem:loglin}, the concentration from Lemma~\ref{lem:var}, and the mean limit from Lemma~\ref{lem:mean} to obtain convergence of the alternate p.g.fl., and hence the Poisson limit.

Fix $u\in\mathcal V$. Throughout, we write $K:=\mathrm{supp}(u)$ (a compact set) and for any set $B$ we define
$B^{+r}:=\{y\in\mathbb R^d:\mathrm{dist}(y,B)\le r\}$.
Note that $u(x)=0$ outside $K$, while the neighbour count $N_{r_n}(x,S_n)$ for $x\in K$
depends only on the restriction of $S_n$ to $K^{+r_n}$.

Recall that conditional on $S_n$, the thinning is independent and each $x\in S_n$ is retained
with probability $p_n(x,S_n)$. In our model
\[
p_n(x,S_n)=\frac 1 n \ a\big(N_{r_n}(x,S_n)\big), x\in S_n.
\]
For $n\ge a_\ast=\sup_{k\in\mathbb N_0} a(k)$, $p_n(x,S_n)\in[0,1]$ for all $x\in S_n$.

The next lemma   provides both the remainder representation
and a \emph{two-sided exponential bound} that will be used in the final step.

\begin{lemma}[Log-linearisation]
\label{lem:loglin}
Let $Y_n$ be obtained from $S_n$ by conditional independent thinning with retention probabilities
$p_n(x,S_n)\in[0,1]$. For $u\in\mathcal V$, set
\[
z_n(x):=u(x)p_n(x,S_n),\qquad
Z_n(u):=\sum_{x\in S_n} z_n(x),\qquad
Q_n(u):=\sum_{x\in S_n} z_n(x)^2.
\]
Then for every $n$ such that $n>2a_\ast:= 2\sup_k a(k)$,
\(
A_{Y_n}(u)
=\mathbb E\exp\{-Z_n(u)+R_n(u)\},\) \(|R_n(u)|\le 2Q_n(u),\)
and moreover
\(\mathbb E e^{-Z_n(u)-2Q_n(u)}
\le A_{Y_n}(u)\le \mathbb E e^{-Z_n(u)}.\)

\end{lemma}

\begin{proof}
Conditional on $S_n$, independence of thinning yields
\[
A_{Y_n}(u\mid S_n)=\mathbb E\left(\prod_{x\in Y_n}(1-u(x))\mid S_n\right)=\prod_{x\in S_n}(1-z_n(x)),
\quad z_n(x)=u(x)p_n(x,S_n)\in[0,1].
\]
For $z\in[0,1/2]$ we have the standard bounds
\(
-z-2z^2 \le \log(1-z)\le -z,\)
 equivalently \(
e^{-z-2z^2}\le 1-z\le e^{-z}.
\)
Under the condition  $n>2a_\ast $  we have $z_n(x)<1/2,$ so 
summing the logarithmic inequality over $x\in S_n$ gives
\[
\log A_{Y_n}(u\mid S_n)=\sum_{x\in S_n}\log(1-z_n(x))
= -\sum_{x\in S_n}z_n(x) + R_n(u),
\]
with $|R_n(u)|\le 2\sum_{x\in S_n}z_n(x)^2=2Q_n(u)$.
Exponentiating  and taking expectations yields the representation
$A_{Y_n}(u)=\mathbb E\exp\{-Z_n(u)+R_n(u)\}$ with the stated bound on $R_n$.
Exponentiating the two-sided bounds on $1-z$ yields
\[
\exp\{-Z_n(u)-2Q_n(u)\}\le A_{Y_n}(u\mid S_n)\le \exp\{-Z_n(u)\}.
\]
Taking expectations completes the proof.
\end{proof}

\begin{lemma}[Local Poissonisation]
\label{lem:localpois}
Fix $x\in\mathbb R^d$ and define
\[
\bar W_n(x):=\sum_{i=2}^n X_i(B(x,r_n)).
\]
Under the   second-order regularity assumption above, continuity of $\lambda$, and the scaling \eqref{eq:crit},
\[
\bar W_n(x)\Rightarrow \mathrm{Poi}\big(\tau\lambda(x)\big), \quad n\to \infty.
\]
\end{lemma}

\begin{proof}
Write $N_n:=X(B(x,r_n))$ for a single copy $X$.
Let $m_n(x):=\mathbb E N_n=\int_{B(x,r_n)}\lambda(y)\,dy$.
By continuity of $\lambda$,
\begin{equation}
\label{eq:mn_asymp}
m_n(x)=\lambda(x)\,v_d r_n^d + o(r_n^d)\qquad (n\to\infty).
\end{equation}

Fix $t\in[0,1]$ and consider $\phi_n(t):=\mathbb E(1-t)^{N_n}$.
For each integer $k\ge 0$, the function $f_k(t)=(1-t)^k$ satisfies
$f_k(0)=1$, $f_k'(0)=-k$, and $f_k''(t)=k(k-1)(1-t)^{k-2}\ge 0$ for $t\in[0,1]$.
Hence by Taylor's theorem with remainder and the bound $0\le (1-\theta t)^{k-2}\le 1$,
\[
0\le (1-t)^k-(1-kt)\le \frac{t^2}{2}k(k-1)=\frac{t^2}{2}(k)_2
\qquad (t\in[0,1]).
\]
Taking expectations with $k=N_n$ yields
\begin{equation}
\label{eq:phi_expansion}
\phi_n(t)=\mathbb E(1-t)^{N_n}
=1-t\,\mathbb EN_n + O\big(\mathbb E(N_n)_2\big),
\qquad \text{uniformly in }t\in[0,1].
\end{equation}
By the small-ball bound  \eqref{eq:smallball},
$\mathbb E(N_n)_2 = O(r_n^{2d})$ as $n\to\infty$.

Since $\bar W_n(x)$ is a sum of $(n-1)$ i.i.d.\ copies of $N_n$,
\[
\mathbb E(1-t)^{\bar W_n(x)} = \phi_n(t)^{\,n-1}.
\]
Combining \eqref{eq:phi_expansion} and \eqref{eq:mn_asymp} gives
\[
\phi_n(t)=1-t\,\lambda(x)\,v_d r_n^d + o(r_n^d) + O(r_n^{2d}).
\]
Multiplying the error term by $(n-1)$, note that
$(n-1)r_n^{2d}=(nr_n^d)\,r_n^d\to 0$ because $nr_n^d$ is bounded and $r_n^d\to 0$.
Therefore,
\[
\phi_n(t)^{\,n-1}
=\Big(1-t\,\lambda(x)\,v_d r_n^d + o(1/n)\Big)^{n-1}
\longrightarrow \exp\{-t\,\tau\lambda(x)\},
\]
which is the alternate p.g.f.\ of $\mathrm{Poi}(\tau\lambda(x))$ at $t$. This proves the claim.
\end{proof}

We now show that $Z_n(u)$ from Lemma \ref{lem:loglin} concentrates around its mean. The key input is locality:
changing one component $X_i$ can only affect (i) the summands corresponding to points of $X_i$,
and (ii) summands corresponding to points within distance $r_n$ of $X_i$ (because only those
neighbour counts can change).

\begin{lemma}[Concentration]
\label{lem:var}
Fix $u\in\mathcal V$ with compact support $K$.
If $nr_n^d$ is bounded, then $\mathrm{Var}(Z_n(u))=O(1/n)$.
\end{lemma}

\begin{proof}
Let $X_1',\dots,X_n'$ be i.i.d.\ copies of $X$, independent of everything, and let
$Z_n^{(i)}$ be defined as $Z_n(u)$ but with $X_i$ replaced by $X_i'$.
By Efron--Stein \cite{EfronStein1981},
\begin{equation}
\label{eq:efron_stein}
\mathrm{Var}(Z_n(u))\le \frac12\sum_{i=1}^n \mathbb E\big(Z_n(u)-Z_n^{(i)}\big)^2.
\end{equation}
By exchangeability it suffices to bound the $i=1$ term.

Write $C:=\|u\|_\infty a_\ast$ and recall
\[
Z_n(u)=\frac1n\sum_{x\in S_n\cap K} u(x)\,a\big(N_{r_n}(x,S_n)\big),\qquad 0\le u\le \|u\|_\infty,\quad 0\le a\le a_\ast.
\]
Changing $X_1$ to $X_1'$ affects only (i) summands corresponding to points of $X_1\cup X_1'$ in $K$,
and (ii) summands corresponding to points of $S_n$ lying within distance $r_n$ of a point of $X_1\cup X_1'$
(because only then the neighbour count can change). Hence, with
\[
D_n:=K\cap \big(X_1\cup X_1'\big)^{+r_n},
\qquad S_{-1}:=\sum_{i=2}^n X_i,
\]
we have the deterministic bound
\begin{equation}\label{eq:delta_bound_short}
\big|Z_n(u)-Z_n^{(1)}\big|
\le \frac{C}{n}\Big(X_1(K)+X_1'(K)+S_{-1}(D_n)\Big).
\end{equation}

We now bound the second moment of the right-hand side of \eqref{eq:delta_bound_short}.
Fix $n$ large so that $r_n\le 1$, hence $K^{+r_n}\subset K^{+1}$.
By local boundedness of $\lambda$ and $\rho^{(2)}$ on $K^{+1}$, there is $M<\infty$ such that for every
bounded Borel $B\subset K^{+1}$,
\begin{equation}
\label{eq:boundedset_second_moment}
\mathbb E X(B)^2\le M(|B|+|B|^2).
\end{equation}
Moreover, $D_n\subset K$ is covered by at most $X_1(K^{+1})+X_1'(K^{+1})$ balls of radius $r_n$, so
\[
|D_n|\le v_d r_n^d\big(X_1(K^{+1})+X_1'(K^{+1})\big).
\]
Conditioning on $D_n$ and using independence of $X_2,\dots,X_n$ together with boundedness of $\mathbb E X(B)^2,$
\[
\mathbb E\!\left[S_{-1}(D_n)^2\,\middle|\,D_n\right]
\le c_1\,n\,|D_n|+c_2\,n^2\,|D_n|^2
\]
for constants $c_1,c_2<\infty$ depending only on local bounds on $K^{+1}$.
Taking expectations and using $\sup_n(nr_n^d)<\infty$ gives
$\sup_n\mathbb E[S_{-1}(D_n)^2]<\infty$.
 Together with \eqref{eq:delta_bound_short} and \eqref{eq:boundedset_second_moment}
,
this yields
\[
\mathbb E\big(Z_n(u)-Z_n^{(1)}\big)^2=O(n^{-2}).
\]  Hence, by \eqref{eq:efron_stein},
\[
\mathrm{Var}(Z_n(u))\le \frac12\,n\cdot O(n^{-2})=O(1/n).
\]

\end{proof}

We now identify the limit of $\mathbb E Z_n(u)$. This is the step where reduced Palm calculus
enters.

\begin{lemma}[Limit of the mean]
\label{lem:mean}
For each $u\in\mathcal V$,
\[
\mathbb{E}Z_n(u)\ \longrightarrow\ \Lambda(u) := \int_{\mathbb{R}^d} u(x)\,\lambda(x)\,\alpha(\tau\lambda(x))\,dx.
\]
\end{lemma}

\begin{proof}
Write
\[
p_n(x,S_n)=\frac{1}{n}\,a\big(N_{r_n}(x,S_n)\big),
\qquad
Z_n(u)=\sum_{x\in S_n}u(x)p_n(x,S_n).
\]
 Let $S^!_{n,x}$ denote a reduced Palm version of $S_n$ at $x$.  By the reduced Campbell--Mecke formula for simple point processes (see, e.g.,
\cite[Corollary~3.1.14]{BaccelliBlaszczyszynKarray2024};
see also \cite[Ch.~13 and Ch.~15]{DaleyVereJones2008Vol2}),
\[
\mathbb E Z_n(u)
=\frac1n\,\mathbb E\sum_{x\in S_n}u(x)\,a\big(N_{r_n}(x,S_n)\big)
=\int_{\mathbb R^d} u(x)\,\lambda(x)\,
 \mathbb E\,a\big(S^!_{n,x}(B(x,r_n))\big)\,dx.
\]

Since $S_n=X_1+\cdots+X_n$ is a superposition of $n$ i.i.d.\ components,  exchangeability implies that
\[
 S^!_{n,x}\ \stackrel{d}{=}\ X^!_x+\sum_{i=2}^n X_i,
\]
where $X^!_x$ is a reduced Palm version of $X$ at $x$, and $X_2,X_3,\dots$ are i.i.d.\ copies of $X$,
independent of $X^{!}_x$. 
Consequently,
\[
 S^!_{n,x}(B(x,r_n))
\ \stackrel{d}{=}\ U_n(x)+\bar W_n(x),
\qquad
U_n(x):=X^!_x(B(x,r_n)),\qquad
\bar W_n(x):=\sum_{i=2}^n X_i(B(x,r_n)).
\]

Using the definition of the reduced Palm distribution together with the second-order product density,
\[
\mathbb E U_n(x)
= \mathbb E X^!_x(B(x,r_n))
=\frac{1}{\lambda(x)}\int_{B(x,r_n)}\rho^{(2)}(x,y)\,dy
=O(r_n^d).
\]
Hence
\[
\mathbb P\big(U_n(x)\ge 1\big)\le \mathbb E U_n(x)\to 0.
\]
Since $a$ is bounded, it follows that
\begin{equation}
\label{eq:own_negligible}
\Big| 
\mathbb E\,a\big(S^!_{n,x}(B(x,r_n))\big)
-\mathbb E\,a\big(\bar W_n(x)\big)
\Big|
\le 2a_\ast\,\mathbb P\big(U_n(x)\ge 1\big)\longrightarrow 0.
\end{equation}

By Lemma~\ref{lem:localpois},
\(
\bar W_n(x)\Rightarrow \mathrm{Poi}\big(\tau\lambda(x)\big).
\)
Since $a$ is bounded on $\mathbb N_0$  and every function on $\mathbb N_0$ is continuous, convergence in distribution yields
\[
\mathbb E\,a\big(\bar W_n(x)\big)
\longrightarrow
\mathbb E\,a\big(\mathrm{Poi}(\tau\lambda(x))\big)
=\alpha(\tau\lambda(x)).
\]
Combining this with \eqref{eq:own_negligible}, we obtain
\[
 \mathbb E\,a\big(S^!_{n,x}(B(x,r_n))\big)
\longrightarrow
\alpha(\tau\lambda(x))
\qquad\text{for each fixed }x.
\]

Finally, since $u$ has bounded support and $\lambda$ is locally bounded, the integrand is dominated by
\(
\|u\|_\infty\,\lambda(x)\,a_\ast\,\mathbf 1_K(x),
\)
which is integrable. Therefore, by dominated convergence,
\[
\mathbb E Z_n(u)
=\int_{\mathbb R^d} u(x)\lambda(x)\,  \mathbb E\,a\big(S^!_{n,x}(B(x,r_n))\big)\,dx
\longrightarrow
\int_{\mathbb R^d} u(x)\lambda(x)\alpha(\tau\lambda(x))\,dx
=\Lambda(u).
\]
\end{proof}

\begin{proof}[Proof of Theorem~\ref{thm:main}]
Fix $u\in\mathcal V$ and apply Lemma~\ref{lem:loglin} with
$p_n(x,S_n)=a(N_{r_n}(x,S_n))/n$. Since $0\le u\le 1$ on its support and $a\le a_\ast$,
we have $\sup_{x,\varphi}u(x)p_n(x,\varphi)\le a_\ast/n\to 0$, so the lemma applies.
By Lemma~\ref{lem:var}, $\mathrm{Var}(Z_n(u))=O(1/n)\to 0$.
By Lemma~\ref{lem:mean}, $\mathbb E Z_n(u)\to \Lambda(u)$.
Hence $Z_n(u)\to \Lambda(u)$ in $L^2$, in particular in probability.
Since $z_n(x)\le a_\ast/n$ and $Q_n(u)=\sum z_n(x)^2$,
\[
Q_n(u)\le \frac{a_\ast}{n}\sum_{x\in S_n} z_n(x)=\frac{a_\ast}{n}Z_n(u).
\]
Taking expectations and using Lemma~\ref{lem:mean} gives
$\mathbb E Q_n(u)\le \frac{a_\ast}{n}\mathbb E Z_n(u)=O(1/n)\to 0$.
In particular, $Q_n(u)\to 0$ in $L^1$ and hence in probability.
By the exponential inequalities in Lemma~\ref{lem:loglin},
\[
\mathbb E e^{-Z_n(u)-2Q_n(u)} \le A_{Y_n}(u)\le \mathbb E e^{-Z_n(u)}.
\]
Because $0\le e^{-Z_n(u)}\le 1$ and $Z_n(u)\to \Lambda(u)$ in probability, the family
$\{e^{-Z_n(u)}\}$ is uniformly integrable, hence
\(
\mathbb E e^{-Z_n(u)}\longrightarrow e^{-\Lambda(u)}.
\)
Moreover,
\[
0\le \mathbb E e^{-Z_n(u)}-\mathbb E e^{-Z_n(u)-2Q_n(u)}
=\mathbb E\big[e^{-Z_n(u)}(1-e^{-2Q_n(u)})\big]
\le \mathbb E(1-e^{-2Q_n(u)})\le 2\mathbb E Q_n(u)\to 0.
\]
Therefore $A_{Y_n}(u)\to e^{-\Lambda(u)}$.

Finally, $e^{-\Lambda(u)}=\exp\{-\int u(x)\tilde\lambda(x)\,dx\}$ is the alternate p.g.fl.\ of an
inhomogeneous Poisson process with intensity density $\tilde\lambda(x)=\lambda(x)\alpha(\tau\lambda(x))$.

Since this convergence holds for every $u\in\mathcal V$, it  holds in particular for continuous
$u$ with bounded support, and hence for the corresponding usual p.g.fl.\ $G_\xi(h)=A_\xi(1-h)$.
By \cite[Proposition~11.1.VIII(iii)]{DaleyVereJones2008Vol2}, this identifies the weak limit of $Y_n$
as that Poisson process, completing the proof.
\end{proof}
\section*{4.\ Discussion and an illustration}

Theorem~\ref{thm:main} shows that the dependent neighbour-count thinning leaves a Poisson limit,
but with a \emph{nonlinear intensity transformation}
\(
\tilde\lambda(x)=\lambda(x)\,\alpha(\tau\lambda(x)),
\) \(
\alpha(t)=\mathbb E\big[a(\mathrm{Poi}(t))\big].
\)
Hence, for fixed $a(\cdot)$ and $\tau$, the effect of interactions is encoded by the scalar response map
$f(\ell):=\ell\,\alpha(\tau\ell)$ acting pointwise on $\ell=\lambda(x)$.
To illustrate possible distortions, we record three concrete choices of bounded neighbour rules $a(k)$
and their resulting $\alpha(\cdot)$.

 \noindent
\textbf{Example 1 (isolated-point rule).}
Let $a(k)=\mathbf 1_{\{k=0\}}$.
Then for $K\sim\mathrm{Poi}(t)$,
\[
\alpha(t)=\mathbb E[a(K)]=\mathbb P(K=0)=e^{-t},
\qquad\text{so}\qquad
\tilde\lambda(x)=\lambda(x)\,e^{-\tau\lambda(x)}.
\]
Thus regions where $\lambda(x)$ is large are strongly suppressed.

 \noindent
\textbf{Example 2 (reciprocal crowding penalty).}
Let $a(k)=1/(k+1)$, which is bounded and decreases with the local neighbour count.
Using the identity $\frac{1}{k+1}=\int_0^1 s^k\,ds$ and $K\sim\mathrm{Poi}(t)$,
\[
\alpha(t)=\mathbb E\Big[\frac{1}{K+1}\Big]
=\int_0^1 \mathbb E[s^K]\,ds
=\int_0^1 \exp(t(s-1))\,ds
=\frac{1-e^{-t}}{t},
\]
with the continuous extension $\alpha(0)=1$.
Consequently,
\[
\tilde\lambda(x)=\lambda(x)\,\frac{1-e^{-\tau\lambda(x)}}{\tau\lambda(x)}
=\frac{1-e^{-\tau\lambda(x)}}{\tau},
\]
so $\tilde\lambda(x)$ saturates at level $1/\tau$ as $\lambda(x)\to\infty$.

 \noindent
\textbf{Example 3 (oscillatory mode).}
Let $a(k)=1+\varepsilon\cos(\omega k)$ with $\varepsilon\in(0,1)$ and $\omega\in(0,\pi)$, so that $a(\cdot)$ is bounded and nonmonotone.
For $K\sim\mathrm{Poi}(t)$ one computes
\[
\alpha(t)=\mathbb E[a(K)]
=1+\varepsilon\,\exp\!\big(t(\cos\omega-1)\big)\cos\!\big(t\sin\omega\big),
\]
hence $\tilde\lambda(x)=\lambda(x)\alpha(\tau\lambda(x))$ exhibits a damped oscillatory modulation as a function of $\tau\lambda(x)$.

\begin{figure}[ht]
\centering
\includegraphics[width=0.7\textwidth]{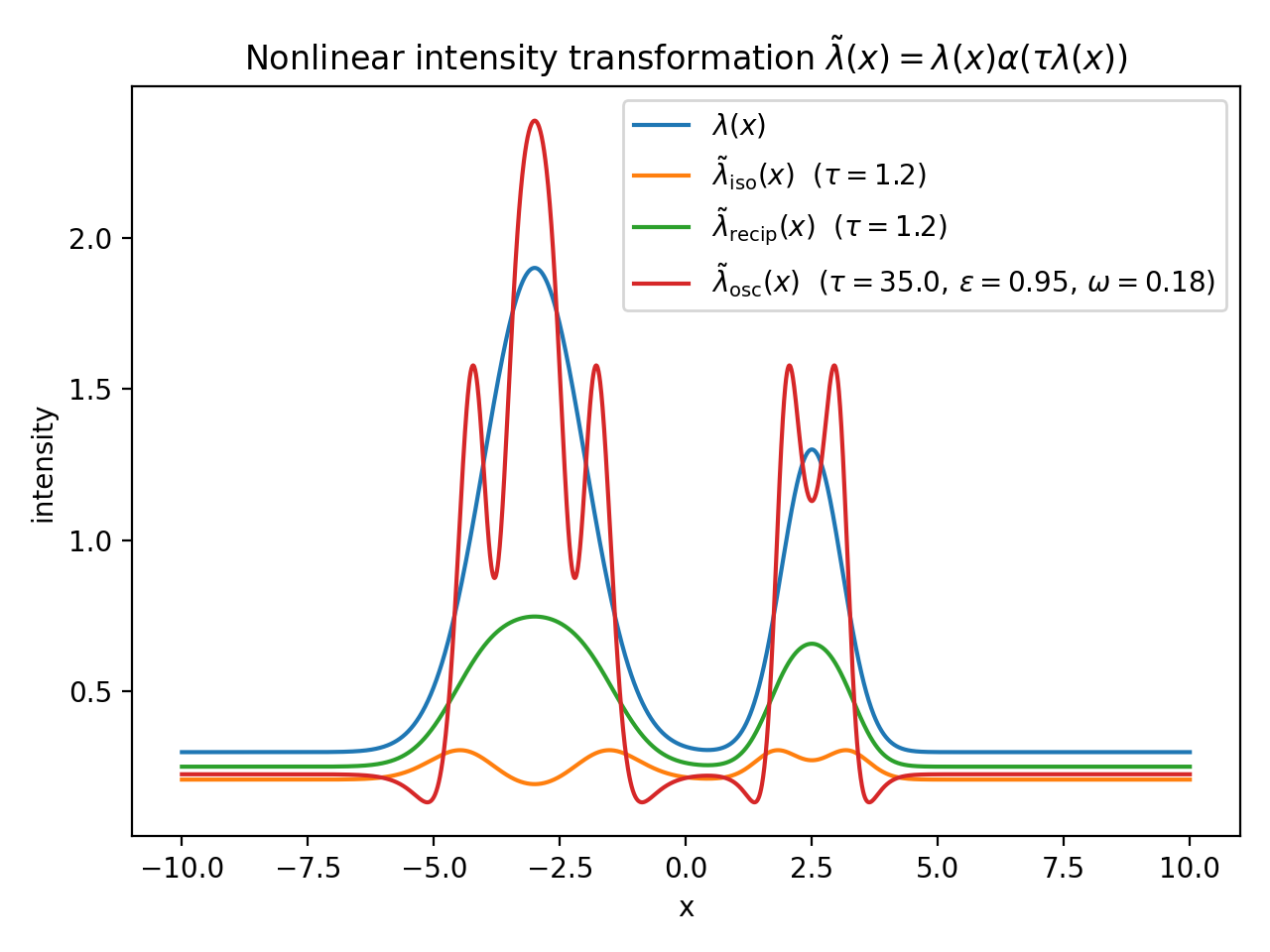}
\caption{Illustration of the nonlinear intensity transformation
$\tilde\lambda(x)=\lambda(x)\alpha(\tau\lambda(x))$ in one dimension for the baseline
$\lambda(x)=0.3+1.6\exp(-(x+3)^2/2)+1.0\exp(-(x-2.5)^2/(2\cdot 0.6^2))$.
We plot $\lambda$ together with (i) $\tilde\lambda_{\mathrm{iso}}(x)=\lambda(x)e^{-\tau\lambda(x)}$
(from $a(k)=\mathbf 1_{\{k=0\}}$, with $\tau=1.2$), (ii)
$\tilde\lambda_{\mathrm{recip}}(x)=(1-e^{-\tau\lambda(x)})/\tau$
(from $a(k)=1/(k+1)$, with $\tau=1.2$), and (iii) 
$\tilde\lambda_{\mathrm{osc}}(x)=\lambda(x)\alpha(\tau\lambda(x))$
from $a(k)=1+\varepsilon\cos(\omega k)$ with parameters $(\tau,\varepsilon,\omega)=(35,0.95,0.18)$.
}
\label{fig:tilde_examples}
\end{figure}

\section*{Acknowledgment}
The author was supported by the NSFC, China grant 12201241 ``Stochastic
flow of Brownian particles with coalescence''.


\begin{thebibliography}{99}

\bibitem{Grigelionis1963Poisson}
B.~Grigelionis,
On the convergence of sums of random step processes to a Poisson process,
\emph{SIAM Theory of Probability \& its Applications} \textbf{8} (2) (1963) 177--182.
\newblock \href{https://doi.org/10.1137/1108017}{https://doi.org/10.1137/1108017}.

\bibitem{Aldridge2026ThinProcesses}
M.~Aldridge,
The law of thin processes: A law of large numbers for point processes,
\emph{Statistics \& Probability Letters} \textbf{232} (2026) 110653.
\newblock \href{https://doi.org/10.1016/j.spl.2026.110653}{https://doi.org/10.1016/j.spl.2026.110653}.




\bibitem{TeichmannBallaniBoogaart2013} J.~Teichmann, F.~Ballani, K.G.~van den Boogaart, Generalizations of Mat\'ern's hard-core point processes, \emph{Spatial Statistics} \textbf{3} (2013) 33--53. \newblock \href{https://doi.org/10.1016/j.spasta.2013.02.001}{https://doi.org/10.1016/j.spasta.2013.02.001}. 

 

\bibitem{Serfozo1984} R.~Serfozo, Thinning of cluster processes: convergence of sums of thinned point processes, \emph{Mathematics of Operations Research} \textbf{9} (4) (1984) 522--533. \newblock \href{https://doi.org/10.1287/moor.9.4.522}{https://doi.org/10.1287/moor.9.4.522}. 



\bibitem{Schuhmacher2009}
D.~Schuhmacher,
Distance estimates for dependent thinnings of point processes with densities,
\emph{Electronic Journal of Probability} \textbf{14} (2009) 1080--1116.
\newblock \href{https://doi.org/10.1214/EJP.v14-643}{https://doi.org/10.1214/EJP.v14-643}.

\bibitem{LastPenrose2018}
G.~Last, M.~Penrose,
\emph{Lectures on the Poisson Process},
Cambridge University Press, 2018.


\bibitem{DaleyVereJones2008Vol2}
D.J.~Daley, D.~Vere-Jones,
\emph{An Introduction to the Theory of Point Processes. Vol.~II: General Theory and Structure},
2nd ed., Springer, New York, 2008.
\newblock \href{https://doi.org/10.1007/978-0-387-49835-5}{https://doi.org/10.1007/978-0-387-49835-5}.





\bibitem{EfronStein1981}
B.~Efron, C.~Stein,
The jackknife estimate of variance,
\emph{The Annals of Statistics} \textbf{9} (3) (1981) 586--596.
\newblock \href{https://doi.org/10.1214/aos/1176345462}{https://doi.org/10.1214/aos/1176345462}.

\bibitem{BaccelliBlaszczyszynKarray2024}
F.~Baccelli, B.~B\l aszczyszyn, M.K.~Karray,
\emph{Random Measures, Point Processes, and Stochastic Geometry},
Inria, 2024.
\newblock HAL preprint:
\href{https://hal.science/hal-02460214v2}{https://hal.science/hal-02460214v2}.






\end{thebibliography}
\end{document}